\pdfoutput=1
\RequirePackage{ifpdf}
\ifpdf 
\documentclass[pdftex]{sigma}
\else
\documentclass{sigma}
\fi

\numberwithin{equation}{section}

\newtheorem{Theorem}{Theorem}[section]
\newtheorem{Proposition}[Theorem]{Proposition}
\newtheorem{Assertion}[Theorem]{Assertion}
{ \theoremstyle{definition}
\newtheorem{Question}[Theorem]{Question}}

\newcommand{\Aaux}{{\text{\usefont{T1}{qcs}{m}{sl}A}}}
\newcommand{\Baux}{{\text{\usefont{T1}{qcs}{m}{sl}B}}}
\newcommand{\Caux}{{\text{\usefont{T1}{qcs}{m}{sl}C}}}
\newcommand{\Eaux}{{\text{\usefont{T1}{qcs}{m}{sl}E}}}
\newcommand{\Kaux}{{\text{\usefont{T1}{qcs}{m}{sl}K}}}
\newcommand{\Maux}{{\text{\usefont{T1}{qcs}{m}{sl}M}}}
\newcommand{\Waux}{{\text{\usefont{T1}{qcs}{m}{sl}W}}}

\newcommand{\Dhom}{{\text{\usefont{T1}{qpl}{m}{sl}D}}}

\newcommand{\aaux}{{\text{\usefont{T1}{qcs}{m}{sl}a}}}
\newcommand{\baux}{{\text{\usefont{T1}{qcs}{m}{sl}b}}}
\newcommand{\caux}{{\text{\usefont{T1}{qcs}{m}{sl}c}}}

\newcommand{\eaux}{{\text{\usefont{T1}{qcs}{m}{sl}e}}}

\newcommand{\kaux}{{\text{\usefont{T1}{qcs}{m}{sl}k}}}
\newcommand{\laux}{{\text{\usefont{T1}{qcs}{m}{sl}l}}}
\newcommand{\maux}{{\text{\usefont{T1}{qcs}{m}{sl}m}}}

\renewcommand{\AA}{\text{\sc a}}
\newcommand{\BB}{\text{\sc b}}
\newcommand{\CC}{\text{\sc c}}
\newcommand{\KK}{\text{\sc k}}
\newcommand{\LL}{\text{\sc l}}
\newcommand{\MM}{\text{\sc m}}

\begin{document}

\allowdisplaybreaks

\newcommand{\arXivNumber}{1906.10880}

\renewcommand{\PaperNumber}{056}

\FirstPageHeading

\ShortArticleName{New Explicit Lorentzian Einstein--Weyl Structures in 3-Dimensions}

\ArticleName{New Explicit Lorentzian Einstein--Weyl Structures\\ in 3-Dimensions}

\Author{Jo\"el MERKER~$^\dag$ and Pawe{\l} NUROWSKI~$^\ddag$}

\AuthorNameForHeading{J.~Merker and P.~Nurowski}

\Address{$^\dag$~Laboratoire de Math\'ematiques d'Orsay, Universit\'e Paris-Sud, CNRS, Universit\'e Paris-Saclay,\\
\hphantom{$^\dag$}~91405 Orsay Cedex, France}
\EmailD{\href{mailto:joel.merker@universite-paris-saclay.fr}{joel.merker@universite-paris-saclay.fr}}
\URLaddressD{\url{http://www.imo.universite-paris-saclay.fr/~merker/}}

\Address{$^\ddag$~Centrum Fizyki Teoretycznej, Polska Akademia Nauk,\\
\hphantom{$^\ddag$}~Al.~Lotnik\'ow 32/46, 02-668 Warszawa, Poland}
\EmailD{\href{mailto:nurowski@cft.edu.pl}{nurowski@cft.edu.pl}}
\URLaddressD{\url{http://www.fuw.edu.pl/~nurowski/}}

\ArticleDates{Received March 30, 2020, in final form June 08, 2020; Published online June 17, 2020}

\Abstract{On a $3$D manifold, a {\it Weyl geometry} consists of pairs $(g, A) =$ (metric, $1$-form) modulo gauge $\widehat{g} = {\rm e}^{2\varphi} g$, $\widehat{A} = A + {\rm d}\varphi$. In 1943, Cartan showed that every solution to the Einstein--Weyl equations $R_{(\mu\nu)} - \frac{1}{3} R g_{\mu\nu} = 0$ comes from an appropriate $3$D leaf space quotient of a $7$D connection bundle associated with a 3\textsuperscript{rd} order ODE $y''' = H(x,y,y',y'')$ modulo point transformations, provided $2$ among $3$ primary point invariants vanish
\begin{gather*}
\text{\sf W\"unschmann}(H) \equiv 0\equiv \text{\sf Cartan}(H).
\end{gather*}
We find that point equivalence of a single PDE $z_y = F(x,y,z,z_x)$ with para-CR integrability $DF := F_x + z_x F_z \equiv 0$ leads to a {\em completely similar} $7$D Cartan bundle and connection. Then magically, the (complicated) equation $\text{\sf W\"unschmann}(H) \equiv 0$ becomes
\begin{gather*}
0\equiv\text{\sf Monge}(F):=9F_{pp}^2F_{ppppp}-45F_{pp}F_{ppp}F_{pppp}+40F_{ppp}^3,\qquad p:=z_x,
\end{gather*}
whose solutions are just conics in the $\{p, F\}$-plane. As an ansatz, we take
\begin{gather*}
F(x,y,z,p):= \frac{\alpha(y)(z-xp)^2\!+\beta(y)(z-xp)p+\gamma(y)(z-xp) +\delta(y)p^2\!+\varepsilon(y)p+\zeta(y)}{\lambda(y)(z-xp)+\mu(y) p+\nu(y)},\!
\end{gather*}
with $9$ arbitrary functions $\alpha, \dots, \nu$ of $y$. This $F$ satisfies $DF \equiv 0 \equiv \text{\sf Monge}(F)$, and we show that the condition $\text{\sf Cartan}(H) \equiv 0 $ passes to a certain $\Kaux(F) \equiv 0$ which holds for any choice of $\alpha(y), \dots, \nu(y)$. Descending to the leaf space quotient, we gain $\infty$-dimensional {\em functionally parametrized and explicit} families of Einstein--Weyl structures $\big[ (g, A) \big]$ in $3$D. These structures are nontrivial in the sense that ${\rm d}A \not\equiv 0$ and $\text{\sf Cotton}([g]) \not \equiv 0$.}


\Classification{83C15; 53C25; 83C20; 53C25; 53C10; 53C25; 53A30; 53A55; 34A26; 34C14; 58A15; 53-08}

\section{Introduction}\label{introduction}

On an $n$-manifold $M$, a {\it Weyl geometry} is a pair $(g, A)$
of a signature $(k, n-k)$ pseudo-Riemannian metric modulo $\widehat{g}
= {\rm e}^{2\varphi} g$ together with a $1$-form $A$ modulo
$\widehat{A} = A + {\rm d}\varphi$, where $\varphi \colon M \to
\mathbb{R}$ is any function. As in Riemannian geometry, a symmetric Ricci
tensor $R_{(\mu\nu)}$ with scalar curvature $R$ can be defined ({\em
see}~\cite{Cartan-1943, Godlinski-2008, Godlinski-Nurowski-2009}
or Section~\ref{Weyl-geometry-summary}).
The {\it Einstein--Weyl equations} in vacuum
\begin{gather}\label{EW-equations-introduction}
R_{(\mu\nu)}-\tfrac{1}{n} R g_{\mu\nu} =0, \qquad 1 \leqslant \mu, \nu \leqslant n,
\end{gather}
which depend only on the class $[(g, A)]$, have raised interest, specially in dimension $n = 3$.
We find various functionally parametrized explicit families of solutions. On $\mathbb{R}^3 \ni (x,y,z)$, take for instance $5$ free arbitrary functions $\baux$, $\caux$, $\kaux$, $\laux$, $\maux$ of $y$ with derivatives $\baux'$, $\kaux'$.

\begin{Theorem}\label{Thm-5-parameterized-Einstein--Weyl-structures}
All pairs $\big(g, A\big)$ such that
\begin{gather*}
g :=(\kaux+\baux z)^2{\rm d}x^2+x^2
\big(\laux^2-\caux\maux\big){\rm d}y^2 +x^2\baux^{2}{\rm d}z^2\\
\hphantom{g :=}{}
+2x
\big(\caux\kaux z-\baux\laux z+\kaux\laux-\baux\maux\big)
{\rm d}x{\rm d}y
-
2x\baux
\big(
\kaux
+
\baux z
\big)
{\rm d}x{\rm d}z
-
2x^2\big(
\caux\kaux
-
\baux\laux
\big)
{\rm d}y{\rm d}z,
\\
A
 :=
\frac{
- \caux\kaux+\baux\laux+\baux'\kaux-\baux\kaux'
}{
x (\caux\kaux^2-2\baux\kaux\laux+\baux^2\maux)}
\big(
x\baux {\rm d}z
-
(\kaux+\baux z) {\rm d}x
\big)\\
\hphantom{A :=}{}
+\frac{\baux\laux^2
-
\caux\baux\maux
-
\baux'\kaux\laux
+
\baux\baux'\maux
+
\caux\kaux\kaux'
-
\baux\kaux'\laux
}{
\caux\kaux^2-2\baux\kaux\laux+\baux^2\maux}
{\rm d}y,
\end{gather*}
satisfy equations~\eqref{EW-equations-introduction},
hence define a Lorentzian Einstein--Weyl structure on $\mathbb{R}^3$.

Moreover, all such examples are generically conformally {\em
non-flat}, and each of the $5$ independent components of the Cotton
tensor of the underlying conformal structure $(M, [g])$ is not
identically zero.
\end{Theorem}

We discover in fact even
more general explicit families of solutions depending on $9$ free
arbitrary functions of $1$ variable $y$.
Explicit examples of Einstein--Weyl structures
in $3$D were known before
\cite{Calderbank-Pedersen-1999,Cartan-1943,Dunajski-Mason-Tod-2001,Eastwood-Tod-2000,
Godlinski-2008,Godlinski-Nurowski-2009, Hitchin-1982,
Jones-Tod-1985,LeBrun-Mason-2009,
Nurowski-2005,
Pedersen-Tod-1993,Tod-1992,
Tod-2000}.

According to~\cite{Cartan-1943},
all Einstein--Weyl structures may be constructed
by a certain quotient process from a $7$D Cartan bundle associated
with equivalences of 3\textsuperscript{rd} order ordinary differential
equations. Those, in turn, are known to be para-CR structures of type
$(1,1,2)$, cf.~\cite[Section~5.1.3]{Hill-Nurowski-2010}.

In the present paper, we explore the observation that PDEs on the
plane $(x,y)$ of the form $z_y = F(x,y,z,z_x)$, considered modulo
point transformations, also happen to be $(1, 1, 2)$ para-CR
structures, in certain circumstances. In
Section~\ref{z-y-F-embedded-in-w-H}, we show how equivalence classes
of $(1,1,2)$ para-CR structures associated to PDEs $z_y = F$ are
`embedded' into the space of
equivalence classes of 3\textsuperscript{rd} order ODEs.
This distinguishes a certain class
of 3\textsuperscript{rd} order ODEs
from which we construct
our explicit solutions to the Einstein--Weyl equations.

Thus, our main approach is to study point equivalences of a single
PDE of the form (novelty)
\[
z_y =F(x,y,z,z_x),
\]
with unknown $z = z(x,y)$. From
para-CR geometry \cite{Hill-Nurowski-2010, Merker-2008}, an
integrability condition is required, namely,
\[
DF :=F_x+z_x F_z \equiv0.
\]
To exclude trivial PDEs,
another point invariant condition must be assumed:
\[
F_{pp}
 \neq 0
\qquad
(\text{\rm abbreviate} \ p := z_x).
\]

In Theorem~\ref{Theorem-Cartan-bundle-z-y-F},
we construct a $7$-dimensional Cartan bundle\big/connection
$P_7 \longrightarrow J_4 \ni (x,y,z,p)$ canonically
associated to point equivalences of such PDEs $z_y = F(x,y,z,z_x)$,
we determine a canonical
coframe $\big\{ \theta^1, \theta^2, \theta^3, \theta^4,\Omega_1, \Omega_2, \Omega_3 \big\}$ on $P_7$,
and we find that its structure
equations~(\ref{structure-equations-GN-2009})
incorporate exactly $3$ primary invariants,
named $\Aaux_1$, $\Baux_1$, $\Caux_1$.

Quite
unexpectedly, we realize that these structure equations
have the same form as the structure equations of the canonical
$7$-dimensional Cartan bundle\big/connection associated
with point equivalences of 3\textsuperscript{rd} order ODEs
$y''' = H(x, y, y', y'')$. Furthermore, it is known that
quite similarly, $3$ primary differential invariants
govern such geometries.
Two among them are: the {\it W\"unschmann invariant}
$\Waux(H)$ \cite{Wunschmann-1905} and the {\it Cartan invariant} $\Caux(H)$~\cite{Cartan-1941,Cartan-1943}.
Since Cartan 1943,
it is also known \cite{Cartan-1943,Fritelli-Kozameh-Newman-2001,Godlinski-2008,
Godlinski-Nurowski-2009, Hitchin-1982} that {\em all} solutions to the
Einstein--Weyl structure equations~(\ref{EW-equations-introduction})
can be obtained from ODEs satisfying $\Waux(H) \equiv 0 \equiv
\Caux(H)$. Translating
what is known for ODEs
or performing computations from scratch,
we will
set up and state Cartan's construction {\em from the PDE side},
{\em see}
Theorem~\ref{Theorem-Wey-g-A-from-ODE}.

But from the ODE side unfortunately,
it is quite difficult to solve W\"unschmann's nonlinear
equation incorporating $25$ differential monomials
\begin{gather*}
0 \equiv\Waux(H) :=
- 18 qH_qH_{pq}
+
9 pH_yH_{qq}
+
18 qHH_{pqq}
+
9 q H_pH_{qq}
-
18 pH_qH_{yq}
+
18 pHH_{yqq}\\
\hphantom{0 \equiv\Waux(H) :=}{}
-
9 HH_qH_{qq}
+
18 pqH_{ypq}
+
18 pH_{xyq}
+
18 q H_{xpq}
+
9 H_xH_{qq}
+
18 HH_{xqq}\\
\hphantom{0 \equiv\Waux(H) :=}{}
-
18 H_qH_{xq}
+
18 H_pH_q
+
9 H_{xxq}
-
27 H_{xp}
+
4 H_q^3
+
9 p^2H_{yyq}
-
27 pH_{yp}\\
\hphantom{0 \equiv\Waux(H) :=}{}
+
9 qH_{yq}
+
9 q^2H_{ppq}
-
27 qH_{pp}
-
18 HH_{pq}
+
9 H^2H_{qqq}
+
54 H_y.
\end{gather*}

This inspired us to try to work on the PDE side $z_y = F(x,y,z,z_x)$,
instead of the ODE side. Then {\em magically}, $\Waux(H) \equiv 0$
transforms into the much simpler classical invariant of Monge~\cite{Monge-1810}
\[
0 \equiv \text{\sf Monge}(F)
 :=
9 F_{pp}^2 F_{ppppp}
-
45 F_{pp} F_{ppp} F_{pppp}
+
40 F_{ppp}^3,
\]
When $F_{pp} \neq 0$, it is known that $\Maux(F) \equiv 0$ holds if
and only if there exist functions $\AA$, $\BB$, $\CC$, $\KK$, $\LL$,
$\MM$ of $(x,y,z)$ such that
\[
0
 \equiv
\AA F^2
+
2\BB Fp
+
\CC p^2
+
2\KK F
+
2\LL p
+
\MM.
\]
Assuming $\AA := 0$, we obtain the following

\begin{Proposition}
The general solution $F = F(x,y,z,p)$ to
\begin{gather*}
0 \equiv F_x+p F_z,\\
0 \equiv 0
+
2\BB Fp
+
\CC p^2
+
2\KK F
+
2\LL p
+
\MM
\end{gather*}
is
\[
F =\frac{\alpha(y)(z-xp)^2+\beta(y)(z-xp)p+\gamma(y)(z-xp)+
\delta(y)p^2+\varepsilon(y)p+\zeta(y)}{
\lambda(y) (z-xp)+\mu(y) p+\nu(y)},
\]
with $9$ arbitrary functions
$\alpha$, $\beta$, $\gamma$, $\delta$, $\varepsilon$, $\zeta$,
$\lambda$, $\mu$, $\nu$ of $y$.
\end{Proposition}

Of course, to the Cartan invariant $\Caux(H)$ from the ODE side
there corresponds from the PDE side
a certain invariant we name $\Kaux(F)$:
its expression appears in
Theorem~\ref{Theorem-Cartan-bundle-z-y-F}.
Miraculously, then, a direct calculation
shows that no further constraint is imposed.

\begin{Proposition}
For any choice of
$\alpha(y)$, $\beta(y)$, $\gamma(y)$, $\delta(y)$, $\varepsilon(y)$,
$\zeta(y)$,
$\lambda(y)$, $\mu(y)$, $\nu(y)$, the second condition
\[
\Kaux\big(F_{\alpha, \dots, \nu}\big)
 \equiv
0
\]
for obtaining Weyl pairs $[(g, A)]$ satisfying the Einstein--Weyl field equations~\eqref{EW-equations-introduction}
holds automatically.
\end{Proposition}

We then get~-- quite long~-- formulas for pairs
$\big[ (g, A) \big]$ expressed explicitly in terms of $\alpha$,
$\beta$, $\gamma$, $\delta$, $\varepsilon$, $\zeta$, $\lambda$, $\mu$,
$\nu$. The subfamily for which $\beta = 0$, $\delta = 0$,
$\varepsilon = 0$, $\mu = 0$ corresponds (with different notations) to
Theorem~\ref{Thm-5-parameterized-Einstein--Weyl-structures}.

\begin{Theorem}\label{Thm-9-parameters-family}
Same conclusion as in
Theorem~{\rm \ref{Thm-5-parameterized-Einstein--Weyl-structures}} with
\begin{gather*}
g :=\tau^1 \tau^2
+
\tau^2 \tau^1
+
\tau^3 \tau^3,
\\
A := \tau^3 \frac{1}{2\Pi}\big(
\gamma\lambda x
-
\gamma\mu
+
x\lambda\nu'
+
\beta\lambda z
+
\lambda\mu' z-
2\alpha\mu z\\
\hphantom{A := \tau^3 \frac{1}{2\Pi}\big(}{}
-
\lambda'\mu z
-
\mu\nu'
-
x\lambda'\nu
-
2x\alpha\nu
+
\beta\nu
+
\mu'\nu
\big),
\end{gather*}
with the coframe
\begin{gather*}
\tau^1
 :=
{\rm d}x
+
\frac{{\rm d}y}{x\lambda-\mu}
\big(x\beta
-
\delta
-
x^2\alpha
\big),
\\
\tau^2
 :=
\frac{2 {\rm d}y}{x\lambda-\mu}
\Pi,
\\
\tau^3 :=
(-\lambda z-\nu)\,
{\rm d}x
+
\frac{1}{x\lambda-\mu}\,
{\rm d}y
\big(
{-}\varepsilon\mu
+
2x^2\alpha\nu
+
x\gamma\mu
-
2x\beta\nu
-
\beta\mu z
+
2\delta\lambda z
+
2x\alpha\mu z\\
\hphantom{\tau^3 :=}{}
+
x\varepsilon\lambda
+
2\delta\nu
-
x^2\gamma\lambda
-
x\beta\lambda z
\big)
+
(x\lambda-\mu)\,
{\rm d}z,
\end{gather*}
and the function
\begin{gather*}
\Pi := x^2\zeta\lambda^2
+
\alpha\mu^2z^2
+
2x\alpha\mu\nu z
+
x^2\alpha\nu^2
-
\beta\lambda\mu z^2
-
x\beta\lambda\nu z
+
\delta\lambda^2 z^2
+
x\varepsilon\lambda^2 z
-
2x\zeta\lambda\mu\\
\hphantom{\Pi :=}{}
-
\beta\mu\nu z
-
x\beta\nu^2
+
2\delta\lambda\nu z
-
\varepsilon\lambda\mu z
+
x\varepsilon\lambda\nu
-
x\gamma\lambda\mu z
-
x^2\gamma\lambda\nu+
\zeta\mu^2
+
\delta\nu^2-
\varepsilon\mu\nu
\\
\hphantom{\Pi :=}{}
+
\gamma\mu^2z
+
x\gamma\mu\nu,
\end{gather*}
again with ${\rm d}A \not\equiv 0$ and $\text{\sf Cotton}
([g]) \not \equiv 0$.
\end{Theorem}

At the end, we also present other families of functionally parametrized solutions, when $\AA \neq 0$.

\section{Weyl geometry: a summary}\label{Weyl-geometry-summary}

In Einstein's theory, gravity is described in terms of a
(pseudo-)riemannian metric $g$ called the {\it gravitational
potential}. In Maxwell's theory, the electromagnetic field is
described in terms of a~$1$-form $A$ called the {\it Maxwell
potential}.

In his attempt {\em Raum, Zeit, Materie}~\cite{Weyl-1919} of unifying gravitation and electromagnetism, Weyl was inspired to introduce the synthetic geometric structure on any $n$-dimensional manifold $M^n$ which consists of classes of such pairs $[(g, A)]$ under the equivalence relation
\[
(g, A) \sim \big(\widehat{g}, \widehat{A}\big)
\]
holding by definition if and only if there exists a function
$\varphi \colon M \longrightarrow \mathbb{R}$ such that
\begin{enumerate}\itemsep=0pt
\item[(1)]$\widehat{g} = {\rm e}^{2 \varphi} g$;
\item[(2)] $\widehat{A} = A + {\rm d}\varphi$.
\end{enumerate}

Clearly, the electromagnetic field strength $F := {\rm d}A$ depends only on the class. The signature $(k, n-k)$ of $g$ can be arbitrary. Conformally Einstein structures from ordinary conformal geometry are a special class of Weyl structures, corresponding to the choice of a closed~-- hence locally exact~-- $1$-form~$A$.

Inspired by Levi-Civita, Weyl established that to such a~{\it Weyl structure} $(M, [(g, A)])$ is associated a~{\em unique} connection $\Dhom$ on $TM$ satisfying:
\begin{enumerate}\itemsep=0pt
\item[(A)] $\Dhom$ has no torsion;
\item[(B)] $\Dhom g = 2 A g$ for any representative $(g, A)$ of the class $[(g, A)]$.
\end{enumerate}

In any (local) coframe $\omega^\mu$, $\mu = 1, \dots, n$, for the
cotangent bundle $T^\ast M$ in which $g = g_{\mu \nu} \omega^\mu
 \omega^\nu$, the connection $1$-forms ${\Gamma^\mu}_\nu$ of
$\Dhom$, or equivalently the $\Gamma_{\mu\nu} := g_{\mu \rho}
{\Gamma^\rho}_\nu$, are indeed uniquely defined from the more
explicit conditions:
\begin{enumerate}\itemsep=0pt
\item[(A$'$)]
${\rm d}\omega^\mu + {\Gamma^\mu}_\nu \wedge \omega^\nu = 0$;
\item[(B$'$)]
$\Dhom g_{\mu \nu} := {\rm d} g_{\mu \nu} - \Gamma_{\mu\nu} - \Gamma_{\nu \mu} = 2 A g_{\mu\nu}$.
\end{enumerate}

Then the {\it curvature} of this Weyl connection identifies with the collection of $n^2$ {\it curvature $2$-forms}
\[
{\Omega^\mu}_\nu
 :=
{\rm d}{\Gamma^\mu}_\nu
+
{\Gamma^\mu}_\rho
\wedge
{\Gamma^\rho}_\nu,
\]
which produce the {\it curvature tensor}
${R^\mu}_{\nu\rho\sigma}$ by expanding in the given coframe
$\omega^\mu$
\[
{\Omega^\mu}_\nu
 =
\tfrac{1}{2}
{R^\mu}_{\nu\rho\sigma}
\omega^\rho
\wedge
\omega^\sigma.
\]

It turns out that ${R^\mu}_{\nu\rho\sigma}$
is a {\it tensor density},
which means in particular that its vanishing
is independent of the choice of a representative $(g, A)$,
and hence as such, serves as a starting point for all invariants
of a Weyl geometry $(M, [(g, A)])$, produced by
covariant differentiation.

Other invariant objects are:
\begin{itemize}\itemsep=0pt
\item the (Weyl--)Ricci tensor $R_{\mu\nu} := {R^\rho}_{\mu\rho\nu}$;
\item its symmetric part $R_{(\mu\nu)} :=\frac{1}{2} ( R_{\mu\nu} + R_{\nu\mu} )$;
\item its antisymmetric part $R_{[\mu\nu]} :=
\frac{1}{2} ( R_{\mu\nu} - R_{\nu\mu} )$.
\end{itemize}

In particular, an appropriately contracted Bianchi identity shows that in $3$-dimensions
\[
R_{[\mu\nu]}
 =
-
\tfrac{3}{2}
F_{\mu\nu},
\]
where $F = {\rm d} A =: \frac{1}{2} F_{\mu \nu}
\omega^\mu \wedge \omega^\nu$.

In~\cite{Cartan-1943}, \'Elie Cartan proposed dynamical
Einstein equations for a Weyl geometry $( M, [(g, A)])$
postulating that the trace-free part of the symmetric Ricci
tensor vanishes
\begin{gather}
\label{Einstein--Weyl-equations}
R_{(\mu\nu)}
-\tfrac{1}{n}
R
g_{\mu\nu}
 =
0,
\end{gather}
where $R := g^{\mu\nu} R_{\mu\nu}$, with
$g^{\mu\rho} g_{\rho\nu} = {\delta^\mu}_\nu$
and $n = \dim M$.

These equations~(\ref{Einstein--Weyl-equations}) are
called {\it Einstein--Weyl equations}, and
a Weyl geometry satisfying~(\ref{Einstein--Weyl-equations}) is called
an {\it Einstein--Weyl structure}. The reason for this
name is as follows.

Since a Weyl structure $( M, [g, A])$ with
vanishing $F = {\rm d}A \equiv 0$
is equivalent to a plain (pseudo-)conformal structure $(M, [g])$
and since the Weyl connection $\Dhom$ then reduces to the
Levi-Civita connection, these
equations~(\ref{Einstein--Weyl-equations})
are a natural generalization of Einstein's field equations.
According to Weyl's approach,
a gravity potential $g$ is thereby {\em coupled} with an electromagnetic field $F = {\rm d}A$.

\section{Cartan's solution to the Einstein--Weyl vacuum equations}\label{Cartan-solution-EW-equations}

In~\cite{Cartan-1941}, Cartan gave a geometric description of
all solutions to the Einstein--Weyl equations~(\ref{Einstein--Weyl-equations})
in $3$-dimensions.
In particular, he showed that there is a {\em one-to-one
correspondence} between 3\textsuperscript{rd}-order
ODEs $y''' = H(x,y,y',y'')$ considered modulo point transformations
of variables which satisfy certain two point-invariant conditions
\begin{gather}
\Waux(H) \equiv 0,\tag{\text{\sf W\"unschmann}}\\
\Caux(H) \equiv 0,\tag{\text{\sf Cartan}}
\end{gather}
and $3$-dimensional Einstein--Weyl structures with
{\em Lorentzian} metrics $g$ of signature $(2,1)$.
Abbreviating $p := y'$, $q := y''$, in terms
of the total differentiation operator
\[
\Dhom
 :=
\partial_x
+
p \partial_y
+
q \partial_p
+
H \partial_q,
\]
their explicit expressions are
\begin{gather}\label{explicit-Wunschmann-W}
\Waux
 :=
9 D D H_q
-
27 D H_p
-
18 H_q D H_q
+
18 H_q H_p
+
4 H_q^3
+
54 H_y,
\\
\label{explicit-Cartan-C}
\Caux
 :=
18 H_{qq} D H_q
-
12 H_{qq} H_q^2
-
54 H_{qq} H_p
+
36 H_{pq} H_q
-
108 H_{yq}
+
54 H_{pp}.
\end{gather}

Although Cartan's geometric arguments \cite{Cartan-1943} offer, in the Lorentzian setting,
a~complete~-- but abstract -- understanding
of the space of all solutions of the Einstein--Weyl
equations~(\ref{Einstein--Weyl-equations}),
it is quite difficult to find {\em explicit} solutions
to the W\"unschmann-Cartan equations $0 \equiv \Waux(H) \equiv
\Caux(H)$, which would provide workable formulas for such
Einstein--Weyl structures.

Some particular solutions are known, e.g.,
\[
H
 =
\frac{3 q^2}{2 p},
\qquad
H
 =
\frac{3 q^2p}{p^2+1},
\qquad
H
 =
q^{3/2},
\qquad
H
 =
\alpha
\frac{
\big(2 qy-p^2\big)^{3/2}}{
y^2},
\qquad
\alpha \in \mathbb{R},
\]
or the `horrible'
\[
H
 =
\frac{pq \big({-}12+3 pq-8 \sqrt{1-pq}\big)
+8 \big(1+\sqrt{1-pq}\big)}{
p^3}.
\]
They were all obtained by rather {\em ad hoc} methods.

In fact, the main difficulty in getting a systematic approach
to finding the solutions is an annoying nonlinearity
of the W\"unschmann condition $\Waux \equiv 0$.

\section{Third-order ODEs modulo point transformations of variables}\label{3-rd-order-ODE-point-transformations}

It was Cartan \cite{Cartan-1941} who solved the equivalence
problem for 3\textsuperscript{rd} order ODEs considered modulo point
transformations. Nowadays, the result may be stated more elegantly in
terms of a certain Cartan connection \cite{Godlinski-2008,Godlinski-Nurowski-2009}, as follows.

To any 3\textsuperscript{rd} order ODE
\begin{align}
\label{y-third-H-x-y-etc}
y'''
 =
H\big(x,y,y',y''\big),
\end{align}
one associates a contact-like coframe on the space
$J_4 \ni (x,y,p,q)$ of $2$-jets of graphs $x \longmapsto y(x)$:
\begin{gather}
\label{initial-coframe-3-rd-ODE}
\omega^1 :={\rm d}y-p \,{\rm d}x,\qquad
\omega^2 :={\rm d}x,\qquad
\omega^3 :={\rm d}p-q\,{\rm d}x,\qquad
\omega^4 :={\rm d}q-H(x,y,p,q)\,{\rm d}x.
\end{gather}
It follows that if a 3\textsuperscript{rd} order ODE~(\ref{y-third-H-x-y-etc})
undergoes a point transformation of variables
\[
(x,y)
 \longmapsto
\big(
\overline{x},
\overline{y}
\big)
 =
\big(
\overline{x}(x,y),
\overline{y}(x,y)
\big),
\]
then the $1$-forms $\big(\omega^1, \omega^2, \omega^3, \omega^4 \big)$
transform as
\begin{gather}
\label{initial-G-structure-3-rd-ODE}
\left(
\begin{matrix}
\omega^1
\\
\omega^2
\\
\omega^3
\\
\omega^4
\end{matrix}
\right)
 \longmapsto
\left(
\begin{matrix}
{\sf u}_1 & 0 & 0 & 0
\\
{\sf u}_2 & {\sf u}_3 & 0 & 0
\\
{\sf u}_4 & 0 & {\sf u}_5 & 0
\\
{\sf u}_6 & 0 & {\sf u}_7 & {\sf u}_8
\end{matrix}
\right)
\left(
\begin{matrix}
\omega^1
\\
\omega^2
\\
\omega^3
\\
\omega^4
\end{matrix}
 \right)
 =:
\left(
\begin{matrix}
\theta^1
\\
\theta^2
\\
\theta^3
\\
\theta^4
\end{matrix}
\right),
\end{gather}
where the ${\sf u}_i$ are certain functions on $J_4$.

Actually, Cartan assures us that the entire equivalence problem
for 3\textsuperscript{rd} order ODEs considered modulo point
transformations of variables is the same as the equivalence
problem for $1$-forms~(\ref{initial-coframe-3-rd-ODE}),
considered modulo
transformations~(\ref{initial-G-structure-3-rd-ODE}).
There is a unique way of reducing these eight group parameters
${\sf u}_i$ to only three ${\sf u}_3$, ${\sf u}_5$, ${\sf u}_7$,
the other ones being expressed in terms of them.
This is achieved by forcing the exterior differentials
of the $\theta^\mu$'s to satisfy the
EDS~(\ref{structure-equations-GN-2009}) below.

\begin{Theorem}[\cite{Cartan-1941, Godlinski-2008, Godlinski-Nurowski-2009}]\label{Theorem-Cartan-connection-3-rd-ODE}
A $3$\textsuperscript{rd} order ODE $y''' = H(x,y,y',y'')$ with its associated $1$-forms
\begin{gather*}
\omega^1
=
{\rm d}y
-
p\,{\rm d}x,
\qquad
\omega^2
=
{\rm d}x,
\qquad
\omega^3
=
{\rm d}p
-
q\,{\rm d}x,
\qquad
\omega^4
=
{\rm d}q
-
H(x,y,p,q)\,
{\rm d}x,
\end{gather*}
uniquely defines a $7$-dimensional fiber bundle $P_7 \longrightarrow J_4$ over the space of
second jets $J_4 \ni (x, y, p, q)$ and a unique coframe
$\big\{ \theta^1, \theta^2, \theta^3, \theta^4, \Omega_1, \Omega_2,
\Omega_3 \big\}$ on $P_7$ enjoying structure equations of the shape
\begin{gather}
{\rm d}\theta^1 =\Omega_1\wedge\theta^1-\theta^2\wedge\theta^3,\nonumber\\
d\theta^2 =
(\Omega_1-\Omega_3)\wedge\theta^2
+
\boxed{\Baux_1\!\!} \, \theta^1\wedge\theta^3
-
\Baux_2 \theta^1\wedge\theta^4,\nonumber\\
{\rm d}\theta^3
 =
\Omega_2\wedge\theta^1
+
\Omega_3\wedge\theta^3
+
\theta^2\wedge\theta^4,\nonumber\\
{\rm d}\theta^4
 =
(2\Omega_3-\Omega_1)\wedge\theta^4
-
\Omega_2\wedge\theta^3
-
\boxed{\Aaux_1\!\!} \, \theta^1\wedge\theta^2,\nonumber\\
{\rm d}\Omega_1
 =
\Omega_2\wedge\theta^2
+
(\Aaux_2-2\Caux_1) \theta^1\! \wedge\theta^2
+
(3\Baux_3+\Eaux_1) \theta^1\!\wedge\theta^3
+
(2\Baux_1-3\Baux_4) \theta^1\wedge\theta^4
+
\Baux_2 \theta^3\wedge\theta^4,\nonumber\\
{\rm d}\Omega_2
 =
\Omega_2\wedge(\Omega_1-\Omega_3)
-
\Aaux_3 \theta^1\wedge\theta^2
+
\Eaux_2 \theta^1\wedge\theta^3
-
(\Baux_3+\Eaux_1) \theta^1\wedge\theta^4
+
\boxed{\Caux_1\!\!}\, \theta^2\wedge\theta^3\nonumber\\
\hphantom{{\rm d}\Omega_2=}{} +
(\Baux_1-2\Baux_4) \theta^3\wedge\theta^4,\nonumber\\
{\rm d}\Omega_3
 =
- \Caux_1 \theta^1\wedge\theta^2
+
(2\Baux_3+\Eaux_1) \theta^1\wedge\theta^3
+
2 (\Baux_1-\Baux_4) \theta^1\wedge\theta^4
+
2\Baux_2 \theta^3\wedge\theta^4.\label{structure-equations-GN-2009}
\end{gather}

Moreover, two equations $y''' = H(x, y, y', y'')$ and
$\overline{y}''' = \overline{H}\big( \overline{x},
\overline{y}, \overline{y}', \overline{y}' \big)$ are locally
point equivalent if and only if there exists a local
bundle isomorphism $\Phi \colon P_7 \overset{\sim}{\longrightarrow}
\overline{P}_7$ between the corresponding bundles
$P_7 \longrightarrow J_4$ and
$\overline{P}_7 \longrightarrow
\overline{J}_4$ satisfying
\begin{gather*}
\Phi^\ast \overline{\theta}^\mu =
\theta^\mu
\qquad
\text{and}
\qquad
\Phi^\ast \overline{\Omega}_i
 =
\Omega_i, \qquad
 \mu = 1, 2, 3, 4,\quad
i = 1, 2, 3 .
\end{gather*}

Exactly $3$ $($boxed$)$ invariants are primary: $\Aaux_1$,
$\Baux_1$, $\Caux_1$, while others express in terms
of them and their covariant derivatives. Point equivalence
to $\overline{y}''' = 0$ is characterized by $0 \equiv \Aaux_1
\equiv \Baux_1 \equiv \Caux_1$.
Two relevant explicit expressions are
\begin{gather}
\Aaux_1
 =
\frac{1}{54}
\frac{{\sf u}_3^3}{{\sf u}_1^3}
\Waux,
\tag{\Waux \ {in} \  (\ref{explicit-Wunschmann-W})}
\\
\Caux_1 =
\frac{1}{54}
\frac{{\sf u}_3}{{\sf u}_1^2}
\left(
\Caux
+
\frac{1}{27}
\Waux_q
\right).
\tag{\Caux \ {in} \ (\ref{explicit-Cartan-C})}
\end{gather}

The seven $1$-forms $\big(\theta^1, \theta^2, \theta^3, \theta^4,
\Omega_1, \Omega_2, \Omega_3 \big)$ set up a Cartan connection
$\widehat{\omega}$ on $P_7$ via
\[
\widehat{\omega}
 :=
\left(
\begin{matrix}
\frac{1}{2}\Omega_1 & \frac{1}{2}\Omega_2 & 0 & 0
\vspace{1mm}\\
-\theta^2 & \Omega_3-\frac{1}{2}\Omega_1 & 0 & 0
\vspace{1mm}\\
\theta^3 & -\theta^4 & \frac{1}{2}\Omega_1-\Omega_3 &
-\frac{1}{2}\Omega_2
\vspace{1mm}\\
2\theta^1 & \theta^3 & \theta^2 & -\frac{1}{2}\Omega_1
\end{matrix}
\right),
\]
and the structure equations~\eqref{structure-equations-GN-2009}
are just the equations for the curvature $\widehat{K}$ of this
connection
\begin{gather*}
{\rm d}\widehat{\omega}
+
\widehat{\omega}\wedge\widehat{\omega}
 =:
\widehat{K}.
\end{gather*}
\end{Theorem}

Now, the structure equations~(\ref{structure-equations-GN-2009})
guarantee that the bundle $P_7$ is foliated by a $4$-dimensional
distribution annihilating the three $1$-forms
$\big( \theta^1, \theta^3, \theta^4 \big)$,
and that the leaf space~$M_3$ of this foliation is equipped
with a natural Weyl geometry, if and only if two among three
primary invariants vanish identically
\[
0 \equiv \Aaux_1(H) \equiv\Caux_1(H).
\]

A representative $(g, A)$ of the concerned Weyl class $[ (g, A)]$ on $M_3$ has then the signature~$(2,1)$ symmetric bilinear form
\[
g :=\theta^3 \theta^3+\theta^1 \theta^4+\theta^4 \theta^1,
\]
which is obtained as the determinant of the
lower-left $2 \times 2$ submatrix of
the connection mat\-rix~$\widehat{\omega}$,
while the $1$-form is defined as
\[
A :=\Omega_3.
\]
It is thanks to the hypothesis $\Aaux_1 \equiv 0 \equiv \Caux_1$
that $g$ and $A$, originally defined on $P_7$, descend on~$M_3$.

Furthermore, according to a result of Cartan in~\cite{Cartan-1943},
any such Weyl geometry $[(g, A)]$
defined on such a leaf space~$M_3$ is {\em automatically
Einstein--Weyl!}

We stress that given $H = H(x, y, p, q)$ satisfying
$\Aaux_1 \equiv 0 \equiv \Caux_1$, or equivalently
\[
\Waux(H)
 \equiv
0
 \equiv
\Caux(H),
\]
one can in principle set up {\em explicit} formulas for the corresponding forms $\theta^1$, $\theta^3$, $\theta^4$, $\Omega_3$ on $P_7$, and this in turn can provide {\em explicit formulas} for $(g, A)$ on $M_3$. However, one substantial obstacle is

\begin{Question}How to solve $\Waux(H) \equiv 0 \equiv \Caux(H)$?
\end{Question}

\section[PDE on the plane $z_y = F(x,y,z,z_x)$ modulo point transformations]{PDE on the plane $\boldsymbol{z_y = F(x,y,z,z_x)}$\\ modulo point transformations}\label{PDE-z_y-F}

We recall that in~\cite{Hill-Nurowski-2010},
it was shown that the equivalence problem for
3\textsuperscript{rd}-order ODEs considered
modulo point transformations of variables
is in one-to-one correspondence
with the equivalence problem
for $4$-dimensional para-CR structures of type
$(1, 1, 2)$, cf.\ also~\cite{Merker-Nurowski-2020,Merker-Pocchiola-2018}.
This thus suggests a new approach for constructing
Lorentzian Einstein--Weyl structures via para-CR structures
of type $(1, 1, 2)$.
Instead of working with general
para-CR structures of type $(1, 1, 2)$, we will concentrate
on a subclass determined in the following way.

We start with a class of PDEs of the form
\[
z_y =F(x,y,z,z_x),
\]
considered modulo point transformations,
for an unknown function $z = z(x,y)$.
We then ask when this class defines a
para-CR structure of type $(1, 1, 2)$.

To answer this (in Proposition~\ref{Prp-answer}),
we need a little preparation. Using the abbreviation
$z_x =: p$, we indeed
consider such PDEs modulo point transformations
of variables
\[
(x,y,z) \longmapsto (\overline{x},\overline{y},\overline{z})
 =
\big(
\overline{x}(x,y,z),
\overline{y}(x,y,z),
\overline{z}(x,y,z)
\big).
\]
This leads to an equivalence problem for the four $1$-forms
\[
\omega_0^1 :={\rm d}z-p\,{\rm d}x-F(x,y,z,p)\,{\rm d}y,\qquad
\omega_0^2 :={\rm d}p,\qquad
\omega_0^3 :={\rm d}x,\qquad
\omega_0^4 :={\rm d}y,
\]
given up to transformations
\begin{gather}\label{omega-zero-4-4-matrix}
\left(
\begin{matrix}
\omega_0^1
\\
\omega_0^2
\\
\omega_0^3
\\
\omega_0^4
\end{matrix}
\right)
 \longmapsto
\left(
\begin{matrix}
{\sf u}_1 & 0 & 0 & 0
\\
{\sf u}_2 & {\sf u}_3 & 0 & 0
\\
{\sf u}_4 & 0 & {\sf u}_5 & {\sf u}_6
\\
{\sf u}_7 & 0 & {\sf u}_8 & {\sf u}_9
\end{matrix}
 \right)
\left(
\begin{matrix}
\omega_0^1
\\
\omega_0^2
\\
\omega_0^3
\\
\omega_0^4
\end{matrix}
 \right).
\end{gather}
Within this coframe $\big\{ \omega_0^1, \omega_0^2, \omega_0^3,
\omega_0^4 \big\}$, in terms of the two operators
\[
D
 :=
\partial_x
+
p \partial_z
\qquad
\text{and}
\qquad
\Delta
 :=
\partial_y
+
F \partial_z,
\]
the exterior differential of any function $F = F(x, y, z, p)$
can be rewritten as
\[
{\rm d}F
 =
F_z \omega_0^1
+
F_p \omega_0^2
+
DF \omega_0^3
+
\Delta F \omega_0^4.
\]

\begin{Proposition}\label{Prp-answer}
The coframe of $1$-forms $\big\{ \omega_0^1, \omega_0^2, \omega_0^3,
\omega_0^4 \big\}$ modulo transformations~\eqref{omega-zero-4-4-matrix}
defines a para-CR structure of type $(1, 1, 2)$ if and only if
\begin{gather*}
0 \equiv DF =F_x+p F_z.
\end{gather*}
\end{Proposition}

\begin{proof}The only nontrivial integrability condition required to constitute a
true para-CR structure comes from
\begin{gather*}
0
 =
{\rm d}\omega_0^1
\wedge
\omega_0^1
\wedge
\omega_0^2
 =- DF\omega_0^1\wedge\omega_0^2\wedge\omega_0^3\wedge\omega_0^4.\tag*{\qed}
\end{gather*}\renewcommand{\qed}{}
\end{proof}

We will now show that for this class of para-CR structures
there is an amazing coincidence between its main invariant,
which will happen to be the Monge invariant with respect to~$p$,
and the classical W\"unschmann invariant
of the corresponding class of 3\textsuperscript{rd}
order ODEs modulo point transformations.

From now on, we will only consider PDEs $z_y = F(x,y,z,z_x)$ satisfying
$DF \equiv 0$. Furthermore, we will also assume that another
point-invariant condition holds
\[
0 \neq F_{pp} \qquad (\text{\rm everywhere}).
\]

Cartan's process leads one to choose more convenient representatives of these forms
\begin{gather*}
\omega^1 :=\omega_0^1,\\
\omega^2 :=\omega_0^2-\frac{\Delta F_{ppp}F_{pp}-\Delta F_{pp}F_{ppp}+3 F_p F_{pp} F_{zpp}
-3 F_{pp}^2 F_{zp}-2 F_p F_{ppp} F_{zp}}{6 F_{pp}^3}\omega_0^1,\\
\omega^3 :=\omega_0^3+F_p \omega_0^4-\frac{1}{3} \frac{F_{ppp}}{F_{pp}}\omega_0^1,\\
\omega^4 :=F_{pp} \omega_0^4+\frac{4 F_{ppp}^2-3 F_{pp} F_{pppp}}{18 F_{pp}^2}\omega_0^1,
\end{gather*}
and we will use this choice in the sequel.

Using Cartan's method, it is then straightforward to solve the equivalence problem for point equivalence classes of such PDEs $z_y = F(x,y,z,z_x)$. The solution is summarized in the following

\begin{Theorem}\label{Theorem-Cartan-bundle-z-y-F}
A PDE system $z_y = F(x,y,z,z_x)$ satisfying the two point-invariant
conditions
\[
DF
 \equiv
0
 \neq
F_{z_xz_x},
\]
with its associated $1$-forms
$\omega^1$, $\omega^2$,
$\omega^3$, $\omega^4$ as above,
uniquely defines a $7$-dimensional principal $H_3$-bundle
$H_3 \longrightarrow P_7 \longrightarrow J_4$
over the space of first jets $J_4 \ni (x, y, z, p)$
with the $($reduced$)$ structure group $H_3$ consisting of matrices
\[
\left(
\begin{matrix}
{\sf u}_3{\sf u}_5 & 0 & 0 & 0
\\
0 & {\sf u}_3 & 0 & 0
\\
-{\sf u}_3{\sf u}_8 & 0 & {\sf u}_5 & 0
\\
-\dfrac{{\sf u}_3{\sf u}_8^2}{2 {\sf u}_5} & 0 & {\sf u}_8 &
\dfrac{{\sf u}_5}{{\sf u}_3}
\end{matrix}
\right),
\qquad
{\sf u}_3 \in \mathbb{R}^\ast,\quad
{\sf u}_5 \in \mathbb{R}^\ast,\quad {\sf u}_8 \in \mathbb{R},
\]
together with a unique coframe $\big\{ \theta^1, \theta^2, \theta^3,
\theta^4, \Omega_1, \Omega_2, \Omega_3 \big\}$ on $P_7$ where
\[
\left(
\begin{matrix}
\theta^1
\\
\theta^2
\\
\theta^3
\\
\theta^4
\end{matrix}
\right)
 :=
\left(
\begin{matrix}
{\sf u}_3{\sf u}_5 & 0 & 0 & 0
\\
0 & {\sf u}_3 & 0 & 0
\\
-{\sf u}_3{\sf u}_8 & 0 & {\sf u}_5 & 0
\\
-\dfrac{{\sf u}_3{\sf u}_8^2}{2 {\sf u}_5} & 0 & {\sf u}_8 &
\dfrac{{\sf u}_5}{{\sf u}_3}
\end{matrix}
 \right)
\left(
\begin{matrix}
\omega^1
\\
\omega^2
\\
\omega^3
\\
\omega^4
\end{matrix}
 \right),
\]
such that the coframe
enjoys precisely the structure
equations~\eqref{structure-equations-GN-2009}.
This time however, the curvature invariants
$\Aaux_1$, $\Aaux_2$, $\Aaux_3$,
$\Baux_1$, $\Baux_2$, $\Baux_3$, $\Baux_4$,
$\Caux_1$, $\Caux_2$, $\Caux_3$,
$\Eaux_1$, $\Eaux_2$
depend on $F = F(x,y,z,p)$ and its derivatives up to order~$6$.

Two relevant explicit expressions are
\[
\Aaux_1 =-\frac{1}{54}\frac{1}{{\sf u}_3^3}\frac{\Maux}{F_{pp}^3},
\qquad
\Caux_1 =\frac{1}{3}\frac{1}{{\sf u}_3^2 {\sf u}_5}
\frac{\Kaux}{F_{pp}^5},
\]
where
\begin{gather*}
\Maux :=9 F_{ppppp} F_{pp}^2-45 F_{pppp} F_{ppp} F_{pp}+40 F_{ppp}^3,\\
\Kaux :=\Delta F_{ppppp} F_{pp}^3-5 \Delta F_{pppp} F_{pp}^2 F_{ppp}
+
12 \Delta F_{ppp}F_{pp}F_{ppp}^2
-
12 \Delta F_{pp}F_{ppp}^3
-
4 \Delta F_{ppp}F_{pp}^2F_{pppp}\\
\hphantom{\Kaux :=}{}
+
9 \Delta F_{pp}F_{pp}F_{ppp}F_{pppp}
-
\Delta F_{pp}F_{pp}^2F_{ppppp}
+
5 F_p F_{pp}^3F_{ppppz}
+
6 F_{pp}^4F_{pppz}\\
\hphantom{\Kaux :=}{}
-
20 F_pF_{pp}^2F_{ppp}F_{pppz}
-
12 F_{pp}^3F_{ppp}F_{ppz}
+
36 F_pF_{pp}F_{ppp}^2F_{ppz}
-
12 F_pF_{pp}^2F_{pppp}F_{ppz}\\
\hphantom{\Kaux :=}{}
+
8 F_{pp}^2F_{ppp}^2F_{pz}
-
24 F_pF_{ppp}^3F_{pz}
-
3 F_{pp}^3F_{pppp}F_{pz}
+
18 F_pF_{pp}F_{ppp}F_{pppp}F_{pz}\\
\hphantom{\Kaux :=}{}
-
2 F_pF_{pp}^2F_{ppppp}F_{pz}.
\end{gather*}

Two equations $z_y = F(x,y,z,z_x)$ and $\overline{z}_{\overline{y}}
= \overline{F} \big( \overline{x}, \overline{y}, \overline{z},
\overline{z}_{\overline{x}} \big)$ satisfying $DF = 0 \neq
F_{z_xz_x}$ and $\overline{D} \overline{F} \equiv 0 \neq
\overline{F}_{ \overline{z}_{\overline{x}}
\overline{z}_{\overline{x}}}$ are locally point equivalent
if and only if there exists a bundle isomorphism
$\Phi \colon P_7 \overset{\sim}{\longrightarrow}
\overline{P}_7$ between the corresponding principal bundles
$H_3 \longrightarrow P_7 \longrightarrow J_4$ and
$\overline{H}_3 \longrightarrow \overline{P}_7 \longrightarrow
\overline{J}_4$ satisfying
\begin{gather*}
\Phi^\ast \overline{\theta}^\mu =\theta^\mu
\qquad
\text{and}
\qquad
\Phi^\ast \overline{\Omega}_i
 =
\Omega_i,
\qquad
\mu = 1, 2, 3, 4,\quad i = 1, 2, 3.
\end{gather*}
\end{Theorem}

This theorem enables one to think about the geometry of a PDE $z_y =
F(x,y,z,z_x)$ with $DF \equiv 0 \neq F_{z_xz_x}$, considered modulo
point transformations of variables, as the geometry of a~certain
3\textsuperscript{rd} order ODE $y''' = H(x,y,y',y'')$, also
considered modulo point transformations. In particular, one can ask
how big is the subclass of point nonequivalent 3\textsuperscript{rd}
order ODEs which are related to PDEs $z_y = F(x,y,z,z_x)$ with $DF
\equiv 0 \neq F_{z_xz_x}$.

We will not answer this question in this paper. Instead, we
concentrate on the Einstein--Weyl geometric aspect of the above
observation.

Since the EDS staying behind the PDEs $z_y = F(x,y,z,z_x)$ with $DF
\equiv 0 \neq F_{z_xz_x}$ is vi\-sib\-ly the same as the EDS for
3\textsuperscript{rd} order ODEs $y''' = H(x,y,y',y'')$, one can look
for PDEs $z_y = F(x,y,z,z_x)$ with $DF \equiv 0 \neq F_{z_xz_x}$,
which in addition satisfy $A_1=C_1=0$, and build a~corresponding
Einstein Weyl geometry, not in terms of $H(x,y,y',y'')$ satisfying
$\Waux(H)\equiv\Caux(H)\equiv 0$, but in terms of the function
$F(x,y,z,z_x)$ satisfying $DF \equiv \Maux(F)\equiv\Kaux(F)\equiv
0$. If only $\Maux(F) \equiv 0$, there exists a conformal Lorentzian
metric on the leaf space of the integrable distribution in $P_7$
annihilated by $\big\{ \theta^1, \theta^3, \theta^4 \big\}$, and when
moreover $\Kaux(F) \equiv 0$, all this produces Einstein--Weyl
geometries. Actually, we gain the following

\begin{Theorem}
\label{Theorem-Wey-g-A-from-ODE}
A PDE $z_y = F(x,y,z,z_x)$ with $DF \equiv 0 \neq F_{z_xz_x}$ defines a bilinear form
$\widetilde{g}$ of signature $(+, +, -, 0$,
$0, 0, 0)$ on the bundle $P_7
\ni (x,y,z,p, u_3, u_5, u_8)$:
\begin{gather*}
\widetilde{g}
 =
\theta^3 \theta^3+\theta^1 \theta^4+\theta^4 \theta^1
 =
\frac{{\sf u}_5^2}{9 F_{pp}^2}
\Big\{
\big(
3 F_{pp} [{\rm d}x+F_p\,{\rm d}y ]
-
F_{ppp} [{\rm d}z-p\,{\rm d}x-F \,{\rm d}y ]
\big)^2
\\
\hphantom{\widetilde{g}=}{}
+
 ({\rm d}z-p\,{\rm d}x-F\,{\rm d}y )
\big(
18 F_{pp}^3 {\rm d}y
+
[4 F_{ppp}^2-3 F_{pp} F_{pppp}]
[
{\rm d}z-p\,{\rm d}x-F\,{\rm d}y
]
\big)
\Big\},
\end{gather*}
degenerate along the rank $4$ integrable distribution
$\mathcal{D}_4$
which is
the annihilator of $\theta^1$, $\theta^3$,
$\theta^4$.

The PDE $z_y = F(x,y,z,z_x)$ with $DF \equiv 0 \neq F_{z_xz_x}$ also defines the $1$-form
\[
\Omega_3
 :=
r_x\,{\rm d}x
+
r_y\,{\rm d}y
+
r_z\,{\rm d}z
+
\tfrac{1}{3}
{\rm d} \big[
\log
\big(
u_5^3
F_{pp}
\big)
\big],
\]
where
\begin{gather*}
r_x =
\frac{1}{3 F_{pp}^4}
\big\{
\Delta F_{ppp}F_{pp}^2
-
\Delta F_{pp}F_{pp}F_{ppp}
+
3 F_pF_{pp}^2F_{ppz}
-
F_{pp}^3F_{pz}
-
2 F_pF_{pp}F_{ppp}F_{pz}
\\
\hphantom{r_x =}{}-
\Delta F_{pppp}F_{pp}^2p
+
3 \Delta F_{ppp}F_{pp}F_{ppp}p
-
3 \Delta F_{pp}F_{ppp}^2 p
+
\Delta F_{pp}F_{pp}F_{pppp} p
\\
\hphantom{r_x =}{}-
4 F_pF_{pp}^2F_{pppz}p
-
2 F_{pp}^3F_{ppz} p
+
9 F_pF_{pp}F_{ppp}F_{ppz}p
+
F_{pp}^2F_{ppp}F_{pz} p
\\
\hphantom{r_x =}{} -
6 F_pF_{ppp}^2F_{pz}p
+
2 F_pF_{pp}F_{pppp}F_{pz}p
\big\},
\\
r_y =
\frac{1}{3 F_{pp}^4}
\big\{
{-}
\Delta F_{pppp}FF_{pp}^2
+
\Delta F_{ppp}F_pF_{pp}^2
-
\Delta F_{pp}F_{pp}^3
+
3 \Delta F_{ppp}F F_{pp}F_{ppp}
\\
\hphantom{r_y =}{}-
\Delta F_{pp}F_pF_{pp}F_{ppp}
-
3 \Delta F_{pp}FF_{ppp}^2
+
\Delta F_{pp}FF_{pp}F_{pppp}
-
4 FF_pF_{pp}^2F_{pppz}
\\
\hphantom{r_y =}{}+
3 F_p^2F_{pp}^2F_{ppz}
-
2 FF_{pp}^3F_{ppz}
+
9 FF_pF_{pp}F_{ppp}F_{ppz}
-
3 F_pF_{pp}^3F_{pz}
\\
\hphantom{r_y =}{}-
2 F_p^2F_{pp}F_{ppp}F_{pz}
+
FF_{pp}^2F_{ppp}F_{pz}
-
6 FF_pF_{ppp}^2F_{pz}
+
2 FF_pF_{pp}F_{pppp}F_{pz}
+
3 F_{pp}^4F_z
\big\},
\\
r_z =
\frac{1}{3 F_{pp}^4}
\big\{
\Delta F_{pppp}F_{pp}^2
-
3 \Delta F_{ppp}F_{pp}F_{ppp}
+
3 \Delta F_{pp}F_{ppp}^2
-
\Delta F_{pp}F_{pp}F_{pppp}+
4 F_pF_{pp}^2F_{pppz}
\\
\hphantom{r_z =}{}
+
2 F_{pp}^3F_{ppz}
-
9 F_pF_{pp}F_{ppp}F_{ppz}
-
F_{pp}^2F_{ppp}F_{pz}
+
6 F_pF_{ppp}^2F_{pz}
-
2 F_pF_{pp}F_{pppp}F_{pz}
\big\}.
\end{gather*}

The degenerate bilinear form $\widetilde{g}$ descends
to a Lorentzian conformal class $[g]$ on the leaf space~$M_3$
of the distribution $\mathcal{D}_4$, if and only if
the Monge invariant $\Maux(F) \equiv 0$ vanishes identically.

When $\Maux(F) \equiv 0$, the local coordinates on $M_3$ are $(x,y,z)$ with the projection
\begin{align*}
P_7& \longrightarrow M_3,\\
\big(x,y,z,p,{\sf u}_3,{\sf u}_5,{\sf u}_8\big)& \longmapsto (x,y,z),
\end{align*}
and the conformal class $[g]$ has a representative which is explicitly expressed in terms of ${\rm d}x$, ${\rm d}y$, ${\rm d}z$, with coefficients depending only on $(x,y,z)$.

Next, $\Omega_3$ descends to a $1$-form denoted $A$ given up to the differential of a function on $M_3 \ni (x,y,z)$, if and only if $\Kaux(F) \equiv 0$.

Moreover, the pair $\big( \widetilde{g}, \Omega_3 \big)$ descends to a~representative of a Weyl structure $[ (g, A)]$ on $M_3$, if and only if both $\Maux(F) \equiv 0$ and $\Kaux(F) \equiv 0$.

Finally, this Weyl structure is actually Einstein--Weyl, namely it satisfies~\eqref{Einstein--Weyl-equations}.
\end{Theorem}

\section[Transformation of the W\"unschmann invariant into the Monge invariant]{Transformation of the W\"unschmann invariant\\ into the Monge invariant}
\label{transformation-Wunschmann-invariant-Monge-invariant}

As we now know, PDEs $z_y = F(x,y,z,z_x)$ with $DF \equiv 0 \neq F_{z_xz_x}$ satisfying $\Aaux_1 \equiv 0 \equiv \Caux_1$ always define an Einstein--Weyl geometry on the leaf
space $M_3$ of the integrable distribution in~$P_7$ annihilated by $\big\{ \theta^1, \theta^3, \theta^4 \big\}$.

The advantage of looking at a Weyl geometry from the
PDE $z_y = F(x,y,z,z_x)$ point of view
rather than from the ODE side $y''' = H(x,y,y',y'')$,
is that now the W\"unschmann invariant of the ODE
becomes the much simpler and classical {\it Monge invariant}
\[
\Aaux_1(H) \sim\Maux(F) =9 F_{pp}^2 F_{ppppp}-45 F_{pp} F_{pppp} F_{ppp}+ 40 F_{ppp}^3.
\]

Serendipitously, the identical vanishing $\Maux(F) \equiv 0$ is well
known to be equivalent to the condition that the graph of $p
\longmapsto F(p)$ is contained in a conic of the $(p,F)$-plane, with
parameters $(x,y,z)$. More precisely,
\begin{gather}\label{conic-p-F}
0 \equiv\Maux(F)
\quad
\Longleftrightarrow
\quad
\AA F^2
+
2 \BB F p
+
\CC p^2
+
2 \KK F
+
2 \LL p
+
\MM
 \equiv
0,
\end{gather}
for some functions $\AA$, $\BB$, $\CC$, $\KK$, $\LL$, $\MM$ depending only on $(x,y,z)$.

Thus, passing from the formulation of Einstein--Weyl's equations in terms of a~3\textsuperscript{rd} order ODE $y''' = H(x,y,y',y'')$ to the formulation in terms of a PDE $z_y = F(x,y,z,z_x)$, we are able to find a rather {\em large class of solutions} to the equation
\[
\Waux(H) \equiv 0.
\]
Indeed, by replacing $\Waux(H) \leadsto \Maux(F)$, the solution~(\ref{conic-p-F}) is just conical!

\section[How to construct new explicit Lorentzian Einstein--Weyl metrics?]{How to construct new explicit Lorentzian Einstein--Weyl\\ metrics?}\label{how-construct-Einstein--Weyl-metrics}

But remember we also have to assure that
\[
0
 \equiv
DF
 =
\partial_xF
+
p \partial_zF.
\]
The simultaneous conditions $DF \equiv 0 \equiv \Maux(F)$
can be achieved for instance by taking $F$ satisfying
\[
\aaux F^2
+
2\baux F (z-px)
+
\caux (z-px )^2
+
2 \kaux F
+
2 \laux (z-px )
+
\maux
 \equiv
0,
\]
with $\aaux$, $\baux$, $\caux$, $\kaux$, $\laux$, $\maux$ being now functions of $y$ {\em only!}

From now on, we will analyze this special solution for $\Maux(F) \equiv 0 \equiv DF$.
The simplest case occurs when avoiding square root by choosing
\[
\aaux
 :=
0,
\]
so that
\begin{gather}\label{fractional-solution-F-b-c-k-l-m}
F
 :=
\frac{
- \caux (z-xp)^2-2\laux (z-xp)-\maux}{
2
\baux (z-xp)
+
2 \kaux}.
\end{gather}
Here
\[
\baux
 =
\baux(y),
\qquad
\caux
 =
\caux(y),
\qquad
\kaux
 =
\kaux(y),
\qquad
\laux
 =
\laux(y),
\qquad
\maux
 =
\maux(y)
\]
are {\em free} arbitrary differentiable functions of one
variable $y$.

A direct check shows that {\em remarkably
this solution~\eqref{fractional-solution-F-b-c-k-l-m}
also satisfies $\Kaux(F) \equiv 0$}!

\begin{Proposition}All such
\[ F := \frac{- \caux (z-xp)^2-2\laux (z-xp)-\maux}{
2 \baux (z-xp) + 2 \kaux}\] with any functions $\baux$, $\caux$,
$\kaux$, $\laux$, $\maux$ of $y$, lead to Einstein--Weyl structures in
$3$-dimensions.
\end{Proposition}

Performing the Cartan procedure to determine the coframe $\big\{
\theta^1, \theta^2, \theta^3, \theta^4, \Omega_1, \Omega_2, \Omega_3
\big\}$, projecting both
$\theta^3
\theta^3 + \theta^1 \theta^4 + \theta^4 \theta^1$
and $\Omega_3$ to the leaf space of the annihilator $M^3$ of
$\big\{ \theta^1, \theta^3, \theta^4 \big\}$, equipping $M_3 \equiv
\mathbb{R}^3$ with coordinates $(x,y,z)$, we therefore obtain {\em
functionally parameterized} Einstein--Weyl structures $\big(g,
A\big)$ on $\mathbb{R}^3 \ni (x,y,z)$ represented by the signature $(2,1)$
Lorentzian metric
\begin{gather*}
g :=
(\kaux+\baux z)^2{\rm d}x^2
+
x^2
\big(
\laux^2
-
\caux\maux
\big)
{\rm d}y^2
+
x^2\baux^{ 2}
{\rm d}z^2
\\
\hphantom{g :=}{}
+
2x
(
\caux\kaux z
-
\baux\laux z
+
\kaux\laux
-
\baux\maux
)
{\rm d}x\,{\rm d}y
-
2x\baux
(
\kaux
+
\baux z
)
{\rm d}x\,{\rm d}z
-
2x^2 (
\caux\kaux
-
\baux\laux
 )
{\rm d}y\,{\rm d}z,
\end{gather*}
together with the differential $1$-form
\begin{gather*}
A
 :=
\frac{
- \caux\kaux+\baux\laux+\baux'\kaux-\baux\kaux'
}{
x \big(\caux\kaux^2-2\baux\kaux\laux+\baux^2\maux\big)}
\big(
x\baux \,{\rm d}z
-
(\kaux+\baux z) \,{\rm d}x
\big)\\
\hphantom{A :=}{}
+
\frac{
\baux\laux^2
-
\caux\baux\maux
-
\baux'\kaux\laux
+
\baux\baux'\maux
+
\caux\kaux\kaux'
-
\baux\kaux'\laux
}{
\caux\kaux^2-2\baux\kaux\laux+\baux^2\maux}\,
{\rm d}y.
\end{gather*}

An independent direct check confirms that equations~(\ref{EW-equations-introduction}) are indeed identically
fulfilled.

As regards the Cotton tensor, we compute its 5 components,
and find that they are not identically zero.
Hence the obtained Einstein--Weyl structures are generically
conformally non-flat. Thus,
Theorem~\ref{Thm-5-parameterized-Einstein--Weyl-structures}
is established. The story for Theorem~\ref{Thm-9-parameters-family} is quite similar.

Next, without assuming $\AA \equiv 0$
in~\eqref{conic-p-F}, let us now make the ansatz that
\begin{gather*}
\aaux F^2
+
2\baux F (z-xp)
+
\caux (z-xp)^2
+
2\kaux F
+
2\laux (z-xp)
+
\maux
 \equiv
0,
\end{gather*}
for some arbitrary functions
$\aaux$, $\baux$, $\caux$, $\kaux$, $\laux$, $\maux$ of $y$.
The (two) solutions $F$
automatically satisfy $DF \equiv 0 \equiv \Maux(F)$.

Since the solutions to Monge's equation are conics in the $(p,F)$-plane, we can rewrite in a~hyperbolic setting
\[
\big(
\aaux F+\baux (z-xp)+\caux
\big)^2
-
\big(
\kaux F+\laux (z-xp)+\maux
\big)^2
 \equiv
1,
\]
with changed functions
$\aaux$, $\baux$, $\caux$, $\kaux$, $\laux$, $\maux$ of $y$.
To avoid transcendental functions in computations,
we parametrize $\cosh t = \frac{1+q^2}{2 q}$ and
$\sinh t = \frac{1-q^2}{2 q}$,
and then, solving for $F$ and for $z-xp$, we may start from
\begin{gather*}
F
 =
\aaux(y)
\frac{1+q^2}{2 q}
+
\baux(y)
\frac{1-q^2}{2 q}
+\caux(y),\\
z-xp
 =
\kaux(y)
\frac{1+q^2}{2 q}
+
\laux(y)
\frac{1-q^2}{2 q}
+
\maux(y),
\end{gather*}
again with (changed) free functions
$\aaux$, $\baux$, $\caux$, $\kaux$, $\laux$, $\maux$ of $y$. Taking
\[
\omega_0^1
 :=
{\rm d}(z-xp)
+
x \,{\rm d}p
-
F \,{\rm d}y,
\qquad
\omega_0^2
 :=
{\rm d}x,
\qquad
\omega_0^3
 :=
{\rm d}y,
\qquad
\omega_0^4
 :=
{\rm d}p,
\]
and performing para-CR Cartan reduction to an
$\{e\}$-structure/connection, we obtain

\begin{Proposition} The second invariant condition $\Kaux(F) \equiv 0$
holds precisely in the following two cases:
\begin{enumerate}\itemsep=0pt
\item[$(1)$] $\kaux = \laux$;
\item[$(2)$]
$\caux = \maux'$ and $\aaux = \frac{\baux \laux + \kaux \kaux' - \laux \laux'}{\kaux}$.
\end{enumerate}
\end{Proposition}

In case (1), we obtain Einstein--Weyl structures for all free functions $\aaux$, $\baux$, $\caux$,
$\laux$, $\maux$ of $y$ given by
\begin{gather*}
g :=2 \tau^1 \tau^2+\big(\tau^3\big)^2,
\qquad
A :=-\frac{2(\aaux+\baux)}{x(\aaux-\baux)\laux}\tau^2-
\frac{\caux-\maux'}{x(\aaux-\baux)}
\tau^3,
\end{gather*}
where
\begin{gather*}
\tau^1
 :=
x(\aaux+\baux)\,{\rm d}y
-
2\laux \,{\rm d}x,
\qquad
\tau^2
 :=
-
\tfrac{1}{2}
x(\aaux-\baux)\,{\rm d}y,
\qquad
\tau^3
 :=
x\caux\,{\rm d}y
-
x\,{\rm d}z
+
(z-\maux)\,{\rm d}x.
\end{gather*}
We verify that these Einstein--Weyl structures have nontrivial $F = {\rm d}A \not\equiv 0$ and nontrivial
$\text{\sf Cotton}([g]) \not\equiv 0$.

In case (2), we obtain Einstein--Weyl structures given by
\[
g
 :=
2 \tau^1 \tau^2
+
\big(\tau^3\big)^2,
\qquad
A
 :=
{\rm d} \big[
\log \big(x^2 \eaux\big)
\big],
\]
where
\begin{gather*}
\tau^1 :=
(\kaux+\laux)\kaux \,{\rm d}x
+
x (\baux\kaux-\baux\laux+\kaux\kaux'-\laux\laux'
 )\,{\rm d}y,
\\
\tau^2
 :=
\tfrac{1}{2}
(\kaux-\laux)\kaux \,{\rm d}x
+
\tfrac{1}{2} x
 (
\baux\kaux-\baux\laux+\kaux\kaux'-\laux\laux'
 )\,{\rm d}y,\\
\tau^3
 :=
- (z-\maux)\kaux\,{\rm d}x
-
x\kaux\maux\,{\rm d}y
+
x\kaux \,{\rm d}z.
\end{gather*}
But this structure, which depends on $3$ functions
$\baux$, $\kaux$, $\laux$ of $y$, is flat
\[
 {\rm d}A
 \equiv
0
 \equiv
\text{\sf Cotton} ([g] ).
\]

Finally, without replacing $p$ by $z-xp$, let us make the ansatz that
\begin{gather*}
\aaux F^2
+
2\baux F p
+
\caux p^2
+
2\kaux F
+
2\laux p
+
\maux
 \equiv
0.
\end{gather*}
Dealing similarly with the hyperbolic case,
\begin{gather*}
F
 =
\aaux(y)
\frac{1+q^2}{2 q}+\baux(y)
\frac{1-q^2}{2 q}+\caux(y),\\
p
 =
\kaux(y)
\frac{1+q^2}{2 q}
+
\laux(y)
\frac{1-q^2}{2 q}
+
\maux(y),
\end{gather*}
we obtain nontrivial Einstein--Weyl structures.
For instance, when $\kaux = \laux$ as in
(1) above
\[
g
 :=
2 \tau^1 \tau^2
+
\big(\tau^3\big)^2,
\qquad
A
 :=
- \frac{\maux'}{(\aaux-\baux)\laux}
\tau^3,
\]
where
\begin{gather*}
\tau^1 :=2\laux\,{\rm d}x+(\aaux+\baux)\,{\rm d}y,
\qquad
\tau^2 :=- \tfrac{1}{2}(\aaux-\baux)\,{\rm d}y,
\qquad
\tau^3 := {\rm d}z -\maux \,{\rm d}x +(\aaux+\baux)\,{\rm d}y.
\end{gather*}
Note that this is again nontrivial
\[
{\rm d}A \not\equiv 0 \not\equiv
\text{\sf Cotton} ([g] ).
\]
and note that we do not have $x$, $z$ dependence here.

\section[Transforming $z_y=F(z_x)$ into $w''' = w'' H(t)$]{Transforming $\boldsymbol{z_y=F(z_x)}$ into $\boldsymbol{w''' = w'' H(t)}$}\label{z-y-F-embedded-in-w-H}

We end up by exploring a link between our PDE systems and 3\textsuperscript{rd} order ODEs.
For simplicity, we will assume that $F = F(z_x)$ depends only on $p = z_x$.

To avoid notational confusion, 3\textsuperscript{rd}-order
ODEs will now be denoted as $w''' = H (t, w, w', w'')$,
and the fundamental $1$-forms as
\begin{alignat*}{3}
&\omega^1 :={\rm d}z-p\,{\rm d}x-F(p)\,{\rm d}y,\qquad && \theta^1 :={\rm d}w-w_1\,{\rm d}t,& \\
&\omega^2 :={\rm d}p,\qquad && \theta^2 :={\rm d}t,& \\
& \omega^3 :={\rm d}x,\qquad && \theta^3 :={\rm d}w_1-w_2\,{\rm d}t,& \\
& \omega^4 :={\rm d}y, \qquad && \theta^4 :={\rm d}w_2-H(t,w,w_1,w_2)\,{\rm d}t.&
\end{alignat*}
We ask what equivalence class of 3\textsuperscript{rd}-order ODE's corresponds to the equivalence class
of PDEs $z_y = F(z_x)$, still with $F_{pp} \neq 0$, and under the $G$-structures of Sections~\ref{3-rd-order-ODE-point-transformations} and~\ref{PDE-z_y-F}.

For this, since $\omega^1$ and $\theta^1$ are both defined up to plain dilations $\omega^1 \sim {\sf u} \omega^1$ and $\theta^1 \sim {\sf u} \theta^1$, we transform $\omega^1$ in order to make the shape of $\theta^1$ appear, using that $F$ depends only on $p$
\[
\omega^1 =
{\rm d}(z-x p-y F(p))
-
(
-x-y F_p(p)
)
\,{\rm d}p
 =:
{\rm d}w
-
w_1\,{\rm d}t
\]
with
$t :=p$, $w := z-x p-y F(p)$, $w_1 :=- x-y F_p(p)$.
With this, $\omega^2 = {\rm d}p = {\rm d}t = \theta^2$.
Next, using $\omega^3 \sim - \omega^3 - \underline{{\sf u} \,{\rm d}y}$, it comes
\begin{gather*}
\omega^3 = {\rm d}x
 = -
\big[{\rm d} (-x-y F_p(p) )
+
y F_{pp}(p)
\,{\rm d}p
+
F_p(p)\,{\rm d}y
\big]
\\
\hphantom{\omega^3 = {\rm d}x}{} \sim
\big[
{\rm d}w_1
+
y F_{pp}(p)
\,{\rm d}p
+
F_p(p)\,{\rm d}y
\big]
-
\underline{F_p(p)\,{\rm d}y}
 =
{\rm d}w_1 - (-y F_{pp}(p))\,{\rm d}p,
\end{gather*}
 whence
$w_2 :=- y F_{pp}(p)$.

A last computation using $\omega^4 \sim {\sf u} \omega^4$
\begin{gather*}
\omega^4
 =
{\rm d}y
 =
-
\frac{1}{F_{pp}(p)}
\big[
{\rm d} (
- y F_{pp}(p)
 )
+
y F_{ppp}(p)\,{\rm d}p
\big]
 \sim
{\rm d}w_2
-
\big({-}y F_{ppp}(p) \,{\rm d}p
\big),
\end{gather*}
shows that the right-hand side function $H = H(t, w_2)$
of the associated ODE $w''' = H$ is independent
of $w, w_1$ as it must be
\[
H := - y F_{ppp}(p) = w_2
\frac{F_{ttt}(t)}{F_{tt}(t)}.
\]
Hence
\[
w''' = w'' \frac{F_{ttt}(t)}{F_{tt}(t)}
\]
is the 3\textsuperscript{rd} order
ODE associated to the para-CR structure given
by $z_y = F(z_x)$.
Observe that $z_y = \frac{1}{2} z_x^2$
becomes $w''' = 0$,
leading to the flat Einstein--Weyl structure.

\begin{Assertion}The W\"unschmann invariant for ODEs $w''' = w''
\frac{F_{ttt}(t)}{F_{tt}(t)}$, where $F(t)$ with
$F_{tt} \neq 0$ is an arbitrary function of one variable,
corresponds to the Monge invariant of the
PDE $z_y = F(z_x)$:
\begin{gather*}
\text{\sf W\"unschmann}
\left(
w_2
\frac{F_{ttt}(t)}{F_{tt}(t)}
\right)
 =
\frac{9 F_{tt}(t) F_{ttttt}(t)
-
45 F_{tt} F_{ttt} F_{tttt}
+
40 F_{ttt}^3}{F_{tt}^3}
 =
\frac{\text{\sf Monge}(F)}{F_{tt}^3}.
\end{gather*}
\end{Assertion}

\begin{proof}
Among the $25$ terms of W\"unschmann's invariant shown
in the Introduction, only $7$ remain thanks to $0 \equiv H_w
\equiv H_{w_1}$:
\begin{gather*}
\text{\sf W\"unschmann}\big(H(t,w_2)\big) =- 9 H H_{w_2} H_{w_2w_2}
- 9 H_t H_{w_2w_2}+18 H H_{tw_2w_2}-18 H_{w_2} H_{tw_2}\\
\hphantom{\text{\sf W\"unschmann}\big(H(t,w_2)\big) =}{}
+9 H_{ttw_2}+4 H_{w_2}^3+9 H^2H_{w_2w_2w_2},
\end{gather*}
and a direct substitution of $H := w_2 \frac{F_{ttt}(t)}{F_{tt}(t)}$ leads to the result.
\end{proof}

\subsection*{Acknowledgements}
Insights of the anonymous referees are gratefully acknowledged.
This collaboration is supported by the National Science
Center, Poland, grant number 2018/29/B/ST1/02583.

\pdfbookmark[1]{References}{ref}
\LastPageEnding

\end{document}